\documentclass[a4paper,11pt]{amsart}

\usepackage[left = 3.5 cm, right = 3.5 cm]{geometry}
\usepackage{amsthm,amssymb,amsmath,amsfonts,mathrsfs,amscd,amsbsy,dsfont,verbatim}
\usepackage{stmaryrd}
\usepackage{graphicx}
\usepackage[all,cmtip]{xy}
\usepackage{longtable}
\usepackage{comment}
\usepackage{babel}
\usepackage{latexsym}
\usepackage{mathtools}
\usepackage[utf8]{inputenc}
\usepackage{xcolor}
\usepackage{fancyhdr}
\usepackage{indentfirst}
\usepackage{paralist}
\usepackage{hyperref}
\usepackage{tikz}
\usetikzlibrary{cd}
\usepackage{enumerate}
\usepackage{multicol}
\usepackage{comment}
\usepackage[all,cmtip]{xy}

\newtheorem{theorem}{Theorem}[section]

\newtheorem{lemma}[theorem]{Lemma}

\theoremstyle{definition}

\newtheorem{remark}[theorem]{Remark}

\newcommand{\Q}{\mathbb Q}
\newcommand{\Z}{\mathbb Z}

\newcommand{\F}{\mathbb F}
\newcommand{\C}{\mathbb C}

\newcommand{\GL}{\text{GL}}
\newcommand{\Id}{\text{Id}}

\DeclareMathOperator{\Gal}{Gal}
\DeclareMathOperator{\Aut}{Aut}

\DeclareMathOperator{\Frob}{Frob}
\DeclareMathOperator{\SL}{SL}

\newcommand{\Oo}{\mathcal{O}}

\newcommand{\mat}[4]{\left(\begin{array}{cc}
#1 &#2\\
#3 &#4
\end{array}\right)}

\title{Wild Galois representations: elliptic curves with wild cyclic reduction}

\author{Nirvana Coppola}

\address{Dipartimento di Matematica, Universit\`a di Padova, Via Trieste 63, 35131 Padova, Italy}
\email{nirvana.coppola@unipd.it}
\begin{document}

\begin{abstract}
In 1990, Kraus \cite{Kraus1990SurAdditive} classified all possible inertia images of the $\ell$-adic Galois representation attached to an elliptic curve over a non-archimedean local field. In \cite{Coppola2020WildField, Coppola2020WildAction}, the author computed explicitly the Galois representation of elliptic curves having non-abelian inertia image, a phenomenon which only occurs when the residue characteristic of the field of definition is $2$ or $3$ and the curve attains good reduction over some non-abelian ramified extension. In this work, the computation of the Galois representation in all the remaining ``wild'' cases, i.e. when the residue characteristic is $p=2$ or $3$ and the curve attains good reduction over an extension whose ramification degree is divisible by $p$ (without assuming the condition on the image of inertia being non-abelian), is completed. This is based on Chapter V of the author's PhD thesis \cite{Coppola2021WildCurves}.
\end{abstract}
\maketitle

\section{Introduction}
Let $K$ be a non-archimedean local field of mixed characteristic, or in other words, up to isomorphism, let $K$ be a finite extension of $\Q_p$, for some rational prime $p$.
Let $E$ be an elliptic curve defined over $K$. For any prime $\ell$, one can consider the $\ell$-adic Tate module $T_\ell(E)$ and the action of the absolute Galois group $\Gal(\overline{K}/K)$ of $K$ on it, which defines a representation denoted by $\rho_{E,\ell}$. It is well-known that $\rho_{E,\ell}$, more precisely its restriction to the absolute inertia subgroup $I_K$ of $\Gal(\overline{K}/K)$, encodes geometric and arithmetic properties of $E$: indeed the image $I=\rho_{E,\ell}(I_K)$ of inertia is
\begin{itemize}
    \item trivial if and only if $E$ has good reduction over $K$, by the Criterion of Néron-Ogg-Shafarevich \cite[Theorem 1]{Serre1968GoodVarieties};
    \item finite if and only if $E$ has potentially good reduction over $K$, as a consequence of the Criterion \cite[Theorem 2]{Serre1968GoodVarieties};
    \item infinite if and only if $E$ has potentially multiplicative reduction.
\end{itemize}

From an arithmetic point of view, the restriction to inertia of $\rho_{E,\ell}$ contains enough information to determine the conductor exponent at $p$ of $E$, a nonnegative integer number which is nonzero if and only if the curve does not have good reduction, and can be defined for more general abelian varieties. Its computation is of interest, to name one example, for the application of the modular method. Indeed if an elliptic curve is modular (e.g. if it is defined over $\Q$, by \cite{Wiles1995ModularTheorem}), then its conductor is equal to the level of the newform to which it is associated.

While the Galois representation in the cases of good and potentially multiplicative reduction were well understood (respectively by point-counting over the reduction of $E$ \cite[IV, \S 1.3]{Serre1997AbelianCurves} and by using the theory of Tate curves \cite[V \S 5]{Silverman1994AdvancedCurves}) prior to the author's PhD work, less was known about the potentially good case. For a survey of all the known cases, see also \cite[\S 2]{Coppola2021WildCurves}.

If $E$ has potentially good reduction, Kraus \cite{Kraus1990SurAdditive} classifies all possible inertia images $I$ in terms of invariants associated to $E$. The possibilities for the image of inertia can be summed up in the following finite list, where the notation is consistent with \cite{DokchitserGroupNames}:
\begin{itemize}
    \item one of the following cyclic groups: $C_2,C_3,C_4,C_6$ (no restrictions on $p$);
    \item $C_3 \rtimes C_4$ (the dicyclic group), only for $p=3$;
    \item $Q_8$ (the quaternions) or $\SL_2(\F_3)$, only for $p=2$.
\end{itemize}

Moreover, by \cite[Theorem 2.10]{Coppola2021WildCurves} or \cite[Corollary 2 to Theorem 2]{Serre1968GoodVarieties}, if $K^{nr}$ is the maximal unramfied extension of $K$, then $E$ acquires good reduction over the minimal field $K^{nr}(E[m])$ where all $m$-torsion of $E$ is defined, for $m\geq 3$ coprime to $p$. One thus has an explicit recipe to determine this extension, and thus the inertia image. Indeed, we have $I \cong \Gal(K^{nr}(E[m])/K^{nr})$, see e.g. \cite[\S 2]{Coppola2020WildField}.

If $E$ has tame reduction, i.e. if it attains good reduction over a tamely ramified extension, the work of Dokchitser and Dokchitser \cite{Dokchitser2016EulerRepresentations} is enough to recover the full Galois representation. A survey of this is presented in \cite{Coppola2021WildCurves}, in a language that is compatible to that used in the present paper.

We say that $E$ has wild reduction if it has potentially good reduction, which it attains over a wildly ramified extension. It follows that $E$ has wild reduction if and only if $p$ divides the order of $I$. The three non-abelian cases $C_3 \rtimes C_4$, $Q_8$ and $\SL_2(\F_3)$, in particular, only occur as wild reduction cases. The large inertia image poses constraints on the whole $\rho_{E,\ell}$, making it relatively simple to determine it. This has been done by the author in \cite{Coppola2020WildField, Coppola2020WildAction}. In this paper, the remaining wild cases are tackled. These are all cyclic, and unexpectedly more subtle than the non-abelian ones. By the list above, we know that if $E$ has wild cyclic reduction and $p=3$, then $I \cong C_3$ or $C_6$, and if $p=2$, then $I \cong C_2, C_4$ or $C_6$. We will reduce these five possibilities to the study of the following two:
\begin{itemize}
    \item $p=3$ and $I \cong C_3$,
    \item $p=2$ and $I \cong C_4$,
\end{itemize}
and give an explicit expression for $\rho_{E,\ell}$. In particular, in all these cases, there exists an explicit finite extension of $K$ over which not only does $E$ attain good reduction, but it is possible to choose, as good models:
\begin{itemize}
    \item $y^2=x^3-x$, when $p=3$, and
    \item $y^2+y=x^3$, when $p=2$.
\end{itemize}

Using this and representation theory of finite groups, the representations are explicitly described in Theorem \ref{thm:wild_cyclic_p3}, for $p=3$, and in Theorem \ref{thm:p2cycinertia}, for $p=2$.

The results presented in this paper are all included in Chapter V of the author's PhD thesis \cite{Coppola2021WildCurves}.

We expect in the foreseeable future to implement this result in MAGMA, to complete the already existing functions ``GaloisRepresentations'' \cite{Coppola2021NirvanaC93/Wild-Galois-Representations-Magma} containing the computation of $\rho_{E,\ell}$ in the non-abelian inertia cases, as in \cite{Coppola2020WildField,Coppola2020WildAction}.

As a final note, we mention the related work of \cite{Dembele2024OnQl}, which aims at computing the so-called ``inertial types'' for elliptic curves defined over $\Q_p$. Although their approach and language is different from the one presented here, the motivation is the same.

\section{Notation}

Throughout the paper, we will use the following notation.
\vspace{8pt}
\begin{center}
\begin{tabular}{r|l}
\hline
    $p$ & a rational prime \\
    $K$ & a $p$-adic field \\
    $v$ & a normalised valuation on $K$ \\
    $\pi$ & a uniformiser of $K$ \\
    $k$ & the residue field of $K$ \\
    $\overline{K}$ & a fixed algebraic closure of $K$ \\
    $K^{nr}$ & the maximal unramified extension of $K$ in $\overline{K}$ \\
    $E$ & an elliptic curve defined over $K$ \\
    $\Delta$ & the discriminant of a given model of $E$ \\
    $E[m]$ & the geometric $m$-torsion of $E$, for any integer $m$ \\
    $\ell$ & a rational prime different from $p$ \\
    $T_\ell (E)$ & the $\ell$-adic Tate module of $E$ \\
    $\rho_{E,\ell}$ & the $\ell$-adic Galois representation of $\Gal(\overline{K}/K)$ on $T_\ell(E)$\\
    $I_K$ & the inertia subgroup of $\Gal(\overline{K}/K)$ \\
    $I$ & the image of inertia $\rho_{E,\ell}(I_K)$\\
\hline
\end{tabular}
\end{center}
\vspace{8pt}

Moreover, we fix an isomorphism $\Aut(T_\ell(E)) \cong \GL_2(\Z_\ell)$ and, if $\overline{\Q}_\ell$ is an algebraic closure of $\Q_\ell$, we fix embeddings $\Q_\ell \hookrightarrow \overline{\Q}_\ell \hookrightarrow \C$. We still denote by $\rho_{E,\ell}$ the following complex representation:
\[
\rho_{E,\ell} : \Gal(\overline{K}/K) \rightarrow \Aut(T_\ell(E)) \cong \GL_2(\Z_\ell) \hookrightarrow \GL_2(\Q_\ell) \hookrightarrow \GL_2(\C).
\]

\section{The case \texorpdfstring{$p=3$}{p=3}}

We assume for this section that $p=3$. Then, by \cite[Theorem 2.7]{Coppola2020WildField} and \cite[Théorème 1]{Kraus1990SurAdditive}, we have that $E$ has wild cyclic reduction exactly when $v(\Delta)$ is even and $E$ has N\'eron type different from $I_0^*$. By \cite[Lemma 2.6]{Coppola2020WildField} and \cite[Corollaire du Lemme 3]{Kraus1990SurAdditive}, this happens if and only if $v(\Delta)$ is even and $E$ has no $2$-torsion points defined over $K^{nr}$. Fix a Weierstrass equation for $E$, of the form $y^2=f(x)$, and let $\alpha_1,\alpha_2,\alpha_3$ be the roots of $f$ in $\overline{K}$. Let $F=K(\alpha_1)$. Then the Galois closure of $F/K$ is $K(E[2])$.

If $v(\Delta) \equiv 0 \pmod 4$, we fix $\Frob_K$ to be the Frobenius element of $\Gal(\overline{K}/K)$ that fixes $F$ point-wise. If $v(\Delta) \equiv 2 \pmod 4$ and $E$ has type different from $I_0^*$, we fix $\Frob_K$ to be the Frobenius element of $\Gal(\overline K/K)$ that fixes $F$ point-wise and a square root $\sqrt{\pi}$ of the uniformiser of $K$.

We will prove the following result.

\begin{theorem}\label{thm:wild_cyclic_p3}
Let $K$ be a $3$-adic field and let $E/K$ have wild cyclic reduction. Let $\chi$ be the unramified character of $\Gal(\overline K/K)$ that sends $\Frob_K$ to $\sqrt{-|k|}$ (which we identify with $i \sqrt{|k|} \in \C$). 
\begin{enumerate}[(a)]
    \item If $v(\Delta) \equiv 0 \pmod 4$ and $[K(E[2]):K]=6$, then $\rho_{E,\ell} = \chi \otimes \psi$, where $\psi$ is the unique irreducible $2$-dimensional representation of $S_3$.
    \item If $v(\Delta) \equiv 0 \pmod 4$ and $[K(E[2]):K]=3$, then:
    \begin{enumerate}[(i)]
        \item if there are $\alpha_i$ and $\alpha_j$ such that $\alpha_i-\alpha_j$ is a square in $K(E[2])$, let $\sigma \in I_K$ be the element of order $3$ that acts on the roots of $f$ as $\sigma(\alpha_j)=\alpha_i$, and let 
        \begin{align*}
  \begin{array}{llll}
       \psi: & I_K &\rightarrow& \overline \Q_\ell^\times  \\
       & \sigma &\mapsto& \dfrac{-1-\sqrt{-3}}{2};
  \end{array}
        \end{align*}
    then 
    \[\rho_{E,\ell}=\chi \otimes \psi \oplus \overline{\chi} \otimes \overline{\psi},\] where $\overline{\bullet}$ denotes complex conjugation;
    \item otherwise, let $\sigma \in I_K$ permute cyclically the roots of $f$, with $\sigma(\alpha_1)=\alpha_2$, and let $\psi$ be defined analogously as in case (b.i). Then 
    \[\rho_{E,\ell}=\overline{\chi} \otimes \psi \oplus \chi \otimes \overline{\psi}.\]
    \end{enumerate}
    \item If $v(\Delta) \equiv 2 \pmod 4$ and $E$ has N\'eron type different from $I_0^*$, let $\chi_\pi$ be the quadratic character of $K(\sqrt{\pi})/K$. Then $E_\pi$, the quadratic twist of $E$ by $K(\sqrt{\pi})$, has wild cyclic reduction and satisfies $v(\Delta) \equiv 0 \pmod 4$. We have 
    \[\rho_{E,\ell}=\rho_{E_\pi,\ell} \otimes \chi_\pi.\]
\end{enumerate}
\end{theorem}

In the rest of this section, we prove each case separately.

\begin{proof}[Proof of case (a)]
First of all, $E$ acquires good reduction over $F$, since $F/K$ is totally ramified of degree $3$ and, by \cite[Theorem 2.7]{Coppola2020WildField} we know that $I=\rho_{E,\ell}(I_K) \cong C_3$. Then, as $[K(E[2]):K]=6$, we have that $K(E[2])/K$ is quadratic and unramified. The same proof as the one for \cite[Theorem 2.12]{Coppola2021WildCurves}, case \textit{(b.ii)}, shows that $\rho_{E,\ell}(\Frob_K)$ has eigenvalues $\pm \sqrt{-|k|}$. Let $\chi$ be as in the statement. Then $\psi = \rho_{E,\ell} \otimes \chi^{-1}$ factors through $\Gal(K(E[2])/K)$, which is isomorphic to $S_3$. Since $\psi\big|_{I_K}$ acts faithfully with image isomorphic to $C_3$, by direct inspection of the $2$-dimensional representations of $S_3$ we deduce that $\psi$ is irreducible, so (as a representation of $S_3$) it is the unique $2$-dimensional irreducible (and faithful) representation of $S_3$.
\end{proof}

For the proof of case \textit{(b.i)}, we assume first that $n=[k:\F_3]$ is odd. At the end of the proof, we will highlight what changes if $n$ is even.

\begin{proof}[Proof of case (b.i).]
Assume that $n=[k:\F_3]$ is odd. Up to relabeling, suppose that $\alpha_2-\alpha_1$ is a square in $F=K(E[2])$, and fix a square root $\sqrt{\alpha_2-\alpha_1} \in F$. The following change of variables
\begin{align*}
    \left\lbrace
    \begin{array}{lll}
        x &=& (\alpha_2-\alpha_1)x'+\alpha_1, \\
        y &=& \sqrt{(\alpha_2-\alpha_1)^3} y'
    \end{array}
\right.
\end{align*}
is defined over $F$ and gives a model for the base change of $E$ to $F$ that reduces to $\tilde{E}: y^2=x^3-x$ over $k$ (for the proof, see \cite[Lemma 3.4]{Coppola2020WildField}). Therefore, by \cite[Example 2.5]{Coppola2020WildField}, the eigenvalues of $\rho_{E,\ell}(\Frob_K)$ are $\pm \sqrt{-|k|}$. We fix $\sigma \in I_K$ that permutes cyclically $\alpha_1,\alpha_2,\alpha_3$ and satisfies $\sigma(\alpha_1)=\alpha_2$, as in the statement. Then $\rho_{E,\ell}(\sigma)$ has determinant $1$ (by the properties of the Weil pairing) and order $3$, so it has eigenvalues given by the two primitive third roots of unity.

The image of $\rho_{E,\ell}$ is abelian, since it is isomorphic to $\Gal(F^{nr}/K)$, which is the direct product of $\Gal(F/K)$ and $\Gal(K^{nr}/K)$. Therefore, we have a splitting of the form $\rho_{E,\ell}=\rho_1 \oplus \rho_2$, where the $\rho_i$'s are $1$-dimensional. From the above, we know that, up to relabeling, $\rho_1(\Frob_K)=\chi(\Frob_K)=\sqrt{-|k|}$, $\rho_2(\Frob_K)=\overline{\chi}(\Frob_K)=-\sqrt{-|k|}$ and $\rho_1(\sigma)$, $\rho_2(\sigma)$ are the two primitive third roots of unity, and it is only necessary to distinguish between the two.

The same argument as in the proof of \cite[Theorem 1.1]{Coppola2020WildField} (more precisely in Section 3.3) can be applied here, and it shows that $\rho_1(\sigma)=\dfrac{-1-\sqrt{-3}}{2}$, and $\rho_2(\sigma)=\dfrac{-1+\sqrt{-3}}{2}$. So, if we define $\psi : I_K \rightarrow \overline{\Q}_\ell^\times$ such that $\psi(\sigma)=\dfrac{-1-\sqrt{-3}}{2}$, we have
\begin{align*}
    \rho_{E,\ell} = \chi \otimes \psi \oplus \overline{\chi} \otimes \overline{\psi},
\end{align*}
as claimed.
\end{proof}

\begin{remark}
If $n=[k:\F_3]$ is even, we still have the same good model for $E$ over $F$, and we can compute the eigenvalues of Frobenius as in Example \cite[Example 2.5]{Coppola2020WildField}, obtaining two identical real values, namely $(-3)^{n/2}$. In this case, we simply have $\rho_{E,\ell}=\chi \otimes (\psi \oplus \overline \psi)$, for $\psi$ defined as in the statement. Since $\chi=\overline{\chi}$, we still recover $\rho_{E,\ell}=\chi \otimes \psi \oplus \overline{\chi} \otimes \overline{\psi}$.
\end{remark}

\begin{proof}[Proof of case (b.ii).]
If all the $\alpha_i-\alpha_j$'s are not squares in $F = K(E[2])$, then $F(\sqrt{\alpha_2-\alpha_1})$ is quadratic and unramified over $F$. In fact, by assumption, $v(\Delta) \equiv 0 \pmod 4$, and if $v_F$ is the normalised valuation on $F$, then $v(\Delta)=v_F(\Delta)/3=2v_F(\alpha_2-\alpha_1)$, since $E$ has potentially good reduction. Therefore, $v_F(\alpha_2-\alpha_1)$ is even, so we can write $\alpha_2-\alpha_1=\pi_F^{2a}\epsilon$, where $\pi_F$ is a uniformiser of $F$, $a \in \Z$ and $\epsilon \in \Oo_F^\times$ is not a square. Moreover, we can take $\pi_F$ and $\epsilon$ so that $\epsilon \in \Oo_K^\times$. Let $E_\epsilon$ be the twist of $E$ be $K(\sqrt{\epsilon})$. Then $E_\epsilon$ is as in case \textit{(b.i)} of the theorem. In fact, since $K(\sqrt{\epsilon})/K$ is unramified, $E_\epsilon$ also has wild cyclic reduction over $K$ with image of inertia $C_3$, and an equation for $E_\epsilon$ is precisely
\begin{align*}
    y^2=(x-\epsilon \alpha_1)(x -\epsilon \alpha_2) (x-\epsilon \alpha_3),
\end{align*}
so $\epsilon\alpha_2-\epsilon\alpha_1 =\epsilon^2 \pi_F^{2a}$ is a square in $F$.

By case \textit{(b.i)}, we have $\rho_{E_\epsilon,\ell} =\chi \otimes \psi \oplus \overline{\chi} \otimes \overline{\psi}$, where $\chi$ and $\psi$ are as in the statement. Let $\eta: \Gal(\overline K/K) \rightarrow \{\pm1\}$ be the unramified quadratic character of $\Gal(\overline K/K)$. Then $\rho_{E_\epsilon,\ell} = \rho_{E,\ell} \otimes \eta$, and an immediate computation shows that
\begin{align*}
    \rho_{E,\ell}= \overline{\chi} \otimes \psi \oplus \chi \otimes \overline{\psi},
\end{align*}
as claimed.
\end{proof}

\begin{proof}[Proof of case (c).]
In this case, by \cite[Theorem 2.7]{Coppola2020WildField}, we have $I=\rho_{E,\ell}(I_K) \cong C_6$. Let $E_\pi$ be the twist of $E$ by $K(\sqrt{\pi})$. Then, the discriminant of $E_\pi$ is equal to $\pi^6 \Delta$, and $v(\pi^6 \Delta)=6+v(\Delta) \equiv 0 \pmod 4$. Therefore, if $\chi_\pi$ is the quadratic character of $\Gal(K(\sqrt{\pi})/K)$, we have $\rho_{E,\ell}=\rho_{E_\pi,\ell} \otimes \chi_\pi$, where $\rho_{E_\pi,\ell}$ is given by one of cases \textit{(a)}, \textit{(b.i)} or \textit{(b.ii)}.
\end{proof}

\section{The case \texorpdfstring{$p=2$}{p=2}}

We assume for this section that $p=2$. Recall that $E/K$ has wild cyclic reduction exactly when the image of inertia $I=\rho_{E,\ell}(I_K)$ is one of $C_2$, $C_4$ or $C_6$, and \cite[Theorem 2.9]{Coppola2021WildCurves}  and \cite[Théorème 2, 3]{Kraus1990SurAdditive} classify these three cases.

We first consider the problem of the restriction to inertia of $\rho_{E,\ell}$ for $I \cong C_2$ or $C_6$. In order to do so, we start by viewing $E$ as an elliptic curve over $K^{nr}$, since the reduction type and the action of inertia do not change. In particular, we have $\rho_{E,\ell}: \Gal(\overline{K}/K^{nr}) \rightarrow \Aut (T_\ell(E))$.
\begin{lemma}\label{lem:p2C2_6_inertia}
Let $E/K^{nr}$ be an elliptic curve with wild cyclic reduction and $I \not\cong C_4$. Then:
\begin{itemize}
    \item if $I \cong C_2$, then $E$ is a quadratic ramified twist of an elliptic curve with good reduction;
    \item if $I \cong C_6$, then $E$ is a quadratic ramified twist of an elliptic curve with tame potentially good reduction.
\end{itemize}
\end{lemma}

The proof is essentially \cite[Proposition 4.3]{Dembele2024OnQl}.

\begin{proof}
Assume first that $I \cong C_2$.

Let $L=K^{nr}(E[3])$. Then, by \cite[Theorem 2.8]{Coppola2020WildField} (or \cite[\S 2 Corollary 3]{Serre1968GoodVarieties}), we have $I\cong \Gal(L/K^{nr})$, so $L/K^{nr}$ is quadratic. Then the quadratic twist $E'$ of $E$ by $L$ has good reduction.

Assume now that $I \cong C_6$ and let $L$ be as above. Then $L/K^{nr}$ is cyclic of order $6$ and there is a unique quadratic subextension $M/K^{nr}$. The quadratic twist of $E$ by $M$ has tame potentially good reduction (achieved over $L$), with inertia image that is cyclic of order $3$.
\end{proof}

We now show that, in fact, there exists a quadratic twist of the base field $K$, over which $E$ acquires good or tame reduction.

\begin{lemma}\label{lem:p2C2_6}
Let $E/K$ be an elliptic curve with wild cyclic reduction and $I \not\cong C_4$. Then:
\begin{itemize}
    \item if $I \cong C_2$ then $E$ is a quadratic ramified twist of an elliptic curve with good reduction;
    \item if $I \cong C_6$ then $E$ is a quadratic ramified twist of an elliptic curve with tame potentially good reduction.
\end{itemize}
\end{lemma}

\begin{proof}
Assume that $I \cong C_2$ and let $L=K^{nr}(E[3])$ as in the proof of Lemma \ref{lem:p2C2_6_inertia}. Then $\Gal(L/K)$ has inertia subgroup isomorphic to $C_2$, and the quotient is the procyclic group $\hat{\Z}$. Therefore, $\Gal(L/K)$ is a semidirect product $C_2 \rtimes \hat{\Z}$, with $C_2$ normal in $\Gal(L/K)$; but then the action of $\hat{\Z}$ can only be trivial, so in fact $\Gal(L/K)= C_2 \times \hat{\Z}$. In particular, we can consider the intermediate extension $F/K$ which is fixed by $\hat{\Z}$: this is Galois, quadratic and totally ramified, with $L/F$ unramified, therefore $E$ acquires good reduction over $F$.

If $I \cong C_6$, let $L$ and $M$ be as in the proof of Lemma \ref{lem:p2C2_6_inertia}, then $\Gal(M/K)$ is isomorphic to the direct product $C_2 \times \hat{\Z}$ as in the previous case, and again by taking $F$ to be the fixed field of $\hat{\Z}$ we conclude.
\end{proof}

Notice that Lemma \ref{lem:p2C2_6} shows the existence of a quadratic twist of $K$ over which the curve acquires good or tame reduction, but it does not give an algorithmic result to compute it. Indeed, if $p=2$, there are several quadratic ramified extensions of $K$ and we need to consider one such that its maximal unramified extension is equal to the field $L$ in the proof of Lemma \ref{lem:p2C2_6} above. Determining explicitly this extension is a non-trivial problem, which we do not tackle here. However, some partial explicit results are available if we restrict to $K=\Q_2$, namely \cite[\S 4.1, Lemma 2]{Freitas2022OnCurves}.

Assuming we have computed a quadratic twist $E'/K$ of $E$ with good or tame reduction, and if $\eta$ is the corresponding quadratic character, then $\rho_{E,\ell} = \rho_{E',\ell} \otimes \eta$, and $\rho_{E',\ell}$ is determined by \cite[Lemma 2.4]{Coppola2020WildField} (or \cite[IV, \S 1.3]{Serre1997AbelianCurves}) if $I\cong C_2$ and \cite[Theorem 2.12]{Coppola2021WildCurves} case \textit{(b)} if $I \cong C_6$.

For the rest of the section, we focus on the remaining wild cyclic case, that is $I \cong C_4$. We fix an arithmetic Frobenius element $\Frob_K$ of $\Gal(\overline{K}/K)$. We will specify which Frobenius we choose when the choice is relevant. Let $n=[k:\F_2]$ be the absolute inertia degree of $K$. We define the following unramified character of $\Gal(\overline{K}/K)$:
\begin{align*}
\begin{array}{ll}
    \chi :& \Gal(\overline K/K) \rightarrow \overline{\Q}_\ell \hookrightarrow \C\\
    & \Frob_K \mapsto (\sqrt{-2})^n \mapsto (i\sqrt{2})^n.
    \end{array}
\end{align*}

Let $G = \Gal(K(E[3])/K)$. Then $G$ is naturally embedded into $\GL_2(\F_3)$, with the embedding given by fixing a basis for $E[3]$ as a $\F_3$-vector space. We will show that $\rho_{E,\ell} \otimes \chi^{-1}$ factors through $G$, and more precisely we will prove the following result.

\begin{theorem}\label{thm:p2cycinertia}
Let $G=\Gal(K(E[3])/K)$ and suppose $I \cong C_4$. Then one of the following holds.
\begin{enumerate}[(a)]
    \item $G \cong C_4$ and $\rho_{E,\ell} = \chi \otimes (\psi \oplus \overline{\psi})$, where 
    \begin{align*}
\begin{array}{ll}
    \psi :& G \rightarrow \overline{\Q}_\ell \hookrightarrow \C\\
    & \sigma \mapsto i
    \end{array}
\end{align*}
for any fixed choice of a generator $\sigma$ of $G$;
\item $G \cong Q_8$ or $D_4$ and $\rho_{E,\ell} = \chi \otimes \psi$, where $\psi$ is the only irreducible faithful $2$-dimensional representation of $G$;
\item $G\cong C_8$ and $\rho_{E,\ell} = \chi \otimes (\psi\oplus\overline{\psi})$, where $\psi$ is the faithful character of $C_8$ that maps $g$ to the $8$-th root of unity $\dfrac{-\sqrt{2}+\sqrt{-2}}{2}$, and $g$, seen as an element of $\GL_2(\F_3)$, is $\mat{2}{2}{2}{1}$.
\end{enumerate}
\end{theorem}

\begin{remark}
Determining which of cases \textit{(a)}, \textit{(b)} or \textit{(c)} occurs can be done, for instance, via \cite[\S 3, Proposition 2 and Lemma 3]{Dokchitser2008Root2}.
\end{remark}

\begin{proof}
By definition, $I = \rho_{E,\ell}(I_K)$ is the image of the absolute inertia subgroup via $\rho_{E,\ell}$. Then we know, by \cite[Theorem 2.10]{Coppola2021WildCurves} or \cite[Corollary 2 to Theorem 2]{Serre1968GoodVarieties}, that $I \cong \Gal(K^{nr}(E[3])/K^{nr})$, so $I$ is isomorphic to the inertia subgroup of $G$. Therefore, $I$ is a normal subgroup of $G$ with cyclic quotient. By the classification in \cite[\S 3, Proposition 2]{Dokchitser2008Root2}, it follows that there are only four possibilities for $G$ to have $I \cong C_4$, namely $G$ is one of $C_4,Q_8,D_4$ or $C_8$. In particular, in each of these cases $K$ contains a third root of the discriminant of $E$, and we have that $n$ is even if $G \cong C_4,Q_8$, while $n$ is odd if $G \cong D_4$ or $C_8$.

Suppose first that $G \cong I \cong C_4$. Then $K(E[3])$ is a quartic extension of $K$, generated by the coordinates of one point of $E$ of order $3$, and following the proof of \cite[Lemma 2.1]{Coppola2020WildAction} we obtain that there exists a model for the base change of $E$ to $K(E[3])$ that reduces to $y^2+y=x^3$ over the residue field $k$. In particular, if we fix $\Frob_K$ to be the arithmetic Frobenius that fixes $K(E[3])$ point-wise, we have that $\rho_{E,\ell}(\Frob_K)$ has eigenvalues $(\pm \sqrt{-2})^n$, and since $n$ is even this means that $\rho_{E,\ell}(\Frob_K)$ is the scalar matrix $(-2)^{n/2} \Id_2$. Therefore, $\rho_{E,\ell} \otimes \chi^{-1}$ factors through $G\cong I$, and as a representation of $I$ it is faithful with trivial determinant. By direct inspection on the character table of the group $C_4$ (see \cite{DokchitserGroupNames}), we deduce that $\rho_{E,\ell} \otimes \chi^{-1}$ is the direct sum of the two one-dimensional faithful representations of $C_4$, hence it is $\psi \oplus \overline{\psi}$ where $\psi$ is as in the statement.

Now suppose that $G \cong Q_8$ or $D_4$. Then $G$ is non-abelian, so by \cite[\S2, Lemma 1]{Dokchitser2008Root2} we have that $\rho_{E,\ell} \otimes \chi^{-1}$ factors through $G$, and as a representation of $G$ it is irreducible and faithful. Since both $Q_8$ and $D_4$ have only one $2$-dimensional irreducible representation (which is also faithful), case \textit{(b)} of the theorem holds.

Finally, we assume that $G \cong C_8$. In this case, since $G$ is cyclic, there exists a unique subextension of $K(E[3])$ of degree $2$ over $K$, namely the unramified extension $K_2$ generated by a primitive third root of unity in $\overline{K}$. Notice that, in particular, this means that $K(E[3])$ is not the compositum of a quartic totally ramified extension of $K$ with a quadratic unramified extension of $K$, because every extension of $K$ contains $K_2$, hence it is not totally ramified.

We have that $I \cong \Gal(K(E[3])/K_2)$, and the restriction $\rho_{E,\ell} \big|_{I_K}$ factors through $I$. By \cite[Figure 4.3]{Robson2017ConnectionsFields} we know that the $3$-division polynomial of $E$ over $K_2$ factors as the product of two quadratic factors, so there are two points $P,Q \in E[3] \setminus \{O\}$ such that the abscissas $x_P,x_Q$ are different but in the same $I$-orbit.

We fix the following generator $\sigma$ of $I$: it is the element that acts on $E[3]$ as
\begin{align*}
    \left\lbrace
    \begin{array}{lll}
         P &\mapsto &Q,  \\
         Q &\mapsto &-P;
    \end{array}
    \right.
\end{align*}
so if we fix $\{P,Q\}$ as a basis for $E[3]$ over $\F_3$ we identify $\sigma$ with the matrix
\begin{align*}
\mat{0}{2}{1}{0}.
\end{align*}

We now want to fix a generator $g$ of $G$, and to do so we fix one of the two elements of order $8$ in $G$ with square equal to $\sigma$, namely $g = \mat{2}{2}{2}{1}$ as in the statement. Let us consider the representation $\rho_{E,\ell} \otimes \chi^{-1}$. Again by the same proof as in \cite[Lemma 2.1]{Coppola2020WildAction}, we have that a model for the base change of $E$ to $K(E[3])$ reduces to $y^2+y=x^3$, thus the Frobenius element of $K(E[3])$, which has inertia degree $2$ over $K$, acts as the scalar matrix $(-2)^{n}\Id_2$. We therefore have that the arithmetic Frobenius of $K$ has distinct eigenvalues $\pm(\sqrt{-2})^n$, so $\rho_{E,\ell} \otimes \chi^{-1}$ factors through $G$. We fix $\Frob_K$ to be the Frobenius element of $K$ that is mapped to $g$ under the quotient map: $\Gal(\overline K/K) \rightarrow G$. Moreover, the restriction to inertia of $\rho_{E,\ell} \otimes \chi^{-1}$ factors through $I \cong C_4$, and as a representation of $C_4$ it is faithful with trivial determinant. Therefore $\rho_{E,\ell}\otimes \chi^{-1} = \psi_a \oplus \psi_b$, where $\psi_a$ and $\psi_b$ are two of the four one-dimensional representations of $C_8$, which are listed in \cite{DokchitserGroupNames}. Namely, the possible representations are denoted in op. cit. by $\rho_3, \rho_5, \rho_6, \rho_8$, and we identify $g$ with the conjugacy class denoted by 8A, and the eighth root of unity $\zeta_8$ with the complex number $e^{2 \pi i /8}= \dfrac{\sqrt{2}+\sqrt{-2}}{2}$. Using that the restriction to inertia has trivial determinant, we deduce that the only possibilities for the set $\{\psi_a,\psi_b\}$ are:
\begin{align*}
    \{\rho_3,\rho_5\}, \ \{\rho_3,\rho_8\}, \ \{\rho_5,\rho_6\}, \ \{\rho_6,\rho_8\},
\end{align*}
and in particular $\psi_b = \overline {\psi_a}$. Let $\psi= \psi_a$. Now we have
\begin{align*}
    \rho_{E,\ell} (\Frob_K) = \chi(\Frob_K) (\psi + \overline{\psi})(g).
\end{align*}

Let us fix $\ell = 3$. Then, the reduction modulo $3$ of $\rho_{E,\ell}$ is equal to the modulo $3$ Galois representation, i.e. the one given by the action of $\Gal(\overline K/K)$ on $E[3]$, so we have that $\rho_{E,\ell}(\Frob_K)$ reduces to $g=\mat{2}{2}{1}{2} \pmod 3$, which has trace $1$ and determinant $2$. A direct computation shows that the only pair in the list above for which this occurs is $\{\rho_6,\rho_8\}$, so the representation $\rho_{E,\ell}$ is given by
\begin{align*}
    \rho_{E,\ell} = \chi \otimes (\rho_6 \oplus \rho_8),
\end{align*}
and this concludes the proof, since $\rho_6(g)=\zeta_8^3=\dfrac{-\sqrt{2}+\sqrt{-2}}{2}$.
\end{proof}

\section*{Acknowledgements}
This work was carried out at the University of Bristol between 2018 and 2021 under the EPSRC studentship No 1961436. The author is currently a member of the INdAM group GNSAGA.

\bibliographystyle{alpha}
\bibliography{references}

\begin{thebibliography}{Cop21b}

\bibitem[Cop20a]{Coppola2020WildAction}
Nirvana Coppola.
\newblock {Wild Galois representations: Elliptic curves over a 2-adic field with non-abelian inertia action}.
\newblock {\em International Journal of Number Theory}, 16(6):1199--1208, 7 2020.

\bibitem[Cop20b]{Coppola2020WildField}
Nirvana Coppola.
\newblock {Wild Galois representations: Elliptic curves over a 3-adic field}.
\newblock {\em Acta Arithmetica}, 195(3):289--303, 2020.

\bibitem[Cop21a]{Coppola2021NirvanaC93/Wild-Galois-Representations-Magma}
Nirvana Coppola.
\newblock {NirvanaC93/Wild-Galois-Representations-Magma}, 2021.

\bibitem[Cop21b]{Coppola2021WildCurves}
Nirvana Coppola.
\newblock {\em {Wild Galois representations of elliptic and hyperelliptic curves}}.
\newblock PhD thesis, University of Bristol, 3 2021.

\bibitem[DD08]{Dokchitser2008Root2}
Tim Dokchitser and Vladimir Dokchitser.
\newblock {Root numbers of elliptic curves in residue characteristic 2}.
\newblock {\em Bulletin of the London Mathematical Society}, 40(3):516--524, 6 2008.

\bibitem[DD16]{Dokchitser2016EulerRepresentations}
Tim Dokchitser and Vladimir Dokchitser.
\newblock {Euler factors determine local Weil representations}.
\newblock {\em Journal fur die Reine und Angewandte Mathematik}, 2016(717):35--46, 8 2016.

\bibitem[DFV24]{Dembele2024OnQl}
Lassina Demb{\'{e}}l{\'{e}}, Nuno Freitas, and John Voight.
\newblock {On Galois inertial types of elliptic curves over Ql}.
\newblock {\em arxiv e-print}, 4 2024.

\bibitem[Dok]{DokchitserGroupNames}
Tim Dokchitser.
\newblock {Group Names}.

\bibitem[FK22]{Freitas2022OnCurves}
Nuno Freitas and Alain Kraus.
\newblock {On the symplectic type of isomorphisms of the p-torsion of elliptic curves}.
\newblock {\em Memoirs of the American Mathematical Society}, 277(1361), 5 2022.

\bibitem[Kra90]{Kraus1990SurAdditive}
Alain Kraus.
\newblock {Sur le d{\'{e}}faut de semi-stabilit{\'{e}} des courbes elliptiques {\`{a}} r{\'{e}}duction additive}.
\newblock {\em Manuscripta Mathematica}, 69(1):353--385, 12 1990.

\bibitem[Rob17]{Robson2017ConnectionsFields}
Wrenna Robson.
\newblock {\em {Connections between torsion points of elliptic curves and reduction over local fields}}.
\newblock PhD thesis, University of Bristol, 2017.

\bibitem[Ser97]{Serre1997AbelianCurves}
Jean-Pierre Serre.
\newblock {\em {Abelian l-Adic Representations and Elliptic Curves}}.
\newblock A K Peters/CRC Press, 11 1997.

\bibitem[Sil94]{Silverman1994AdvancedCurves}
Joseph~H. Silverman.
\newblock {\em {Advanced Topics in the Arithmetic of Elliptic Curves}}, volume 151 of {\em Graduate Texts in Mathematics}.
\newblock Springer New York, New York, NY, 1994.

\bibitem[ST68]{Serre1968GoodVarieties}
Jean-Pierre Serre and John Tate.
\newblock {Good Reduction of Abelian Varieties}.
\newblock {\em The Annals of Mathematics}, 88(3):492, 11 1968.

\bibitem[Wil95]{Wiles1995ModularTheorem}
Andrew Wiles.
\newblock {Modular Elliptic Curves and Fermat's Last Theorem}.
\newblock {\em The Annals of Mathematics}, 141(3):443, 5 1995.

\end{thebibliography}
\end{document}